\documentclass[11pt,a4paper,reqno]{amsart}
\usepackage[latin1]{inputenc}
\usepackage[english]{babel}
\usepackage{amsmath}
\usepackage{amsfonts}
\usepackage{amssymb}
\usepackage{mathrsfs}
\usepackage{latexsym}
\usepackage{yfonts}
\usepackage{natbib}

\usepackage[margin=1.8cm]{geometry}

\usepackage{color}

\usepackage{graphicx,color}

\usepackage[plainpages=false,colorlinks,hyperindex,bookmarksopen,linkcolor=red,citecolor=blue,urlcolor=blue]{hyperref}

\bibpunct{[}{]}{;}{n}{,}{,}

\newtheorem{theorem}{Theorem}[section]
\newtheorem{definition}{Definition}[section]

\newtheorem{remark}{Remark}[section]

\numberwithin{equation}{section}

\allowdisplaybreaks

	\author{Raffaela Capitanelli}
	\address{Department of Basic and Applied Sciences for Engineering
	\newline Sapienza University of Rome
	\newline via A. Scarpa 10, Rome, Italy}
	\email{raffaela.capitanelli@uniroma1.it}

	\author{Mirko D'Ovidio}
	\address{Department of Basic and Applied Sciences for Engineering
	\newline Sapienza University of Rome
	\newline via A. Scarpa 10, Rome, Italy}
	\email[corresponding author]{mirko.dovidio@uniroma1.it}

\title{Fractional equations via convergence of forms}
\date{\today}

\begin{document}

\maketitle

\begin{abstract}
We relate the convergence of time-changed processes driven by fractional equations to the convergence of corresponding Dirichlet forms. The fractional equations we dealt with are obtained by considering a general fractional operator in time. 
\end{abstract}



\normalsize

\section{Introduction}

Time fractional derivative is usually considered in the Riemann-Liouville or in the Caputo sense. The Caputo derivative, for $\beta \in (0,1)$, is defined as
\begin{align*}
D_t^\beta u(t) : = \frac{1}{\Gamma(1-\beta)} \int_0^t \frac{u^\prime(s)}{(t-s)^\beta}ds = {^*D}_t^\beta \big(u(t) - u(0)\big)
\end{align*}
where $\Gamma$ is the Gamma function and 
\begin{align*}
{^*D}_t^\beta u(t) := D_t \left[ \frac{1}{\Gamma(1-\beta)} \int_0^t \frac{u(s)}{(t-s)^\beta} ds \right],
\end{align*}
with $D_t=d/dt$, is the Riemann-Liouville derivative ($D^1_t = {^*D}^1_t = D_t$), see for example \cite{KST}. Researchers started considering time fractional equations in order to model anomalous diffusions, that is, diffusions with a non-linear relationship to time in which the mean squared displacement is proportional to a power $\beta \neq 1$ of time.  For $\beta \in (0,1)$, the anomalous diffusion exhibits a subdiffusive behavior (for example, due to particle sticking and trapping phenomena) whereas, for $\beta >1$ we have superdiffusive behaviour (for instance, jumps). Such anomalous feature can be also found in transport phenomena in complex systems, e.g. in random fractal structures (see \cite{GionaRoman}, \cite{MS-book} and \cite{Barlow}). 

We study fractional equations, that is equations involving fractional operators in time and in space.  A type of space fractional equation we consider in this work is the one involving the fractional Laplacian, for instance. Equations involving time-space fractional derivative have been considered in order to model ground water flows and transport (\cite{SMZZC}, fractional in space and example of superdiffusion), the motion of individual fluorescently labeled mRNA molecules inside live E. coli cells (\cite{GoldCox}, fractional in time and example of subdiffusion). Further applications can be found in the theory of viscoelasticity (\cite{AppFC-Soc}), in modeling the cardiac tissue electrode interface (\cite{AppFC-Mag}), in modeling the anisotropies of the Cosmic Microwave Background radiation where the involved processes move on the sphere (\cite{DovSPA}, \cite{DovJSP}) or, in general, on compact manifolds (\cite{DovNan16}). 

Time fractional equations has been treated by a number of authors. In  \cite{SWyss89}, \cite{Wyss86} and later in \cite{Hil00} the solutions to the heat-type fractional diffusion equation have been studied and presented in terms of Fox's functions. The same equations have been investigated in \cite{OB03,OB09} where the solutions have been represented by means of stable densities, focusing on the explicit representations only in some cases. Different boundary value problems have been studied for example  in \cite{BM01, BO09elastic, KLW, MezK00} following different approaches. In the papers \cite{MLP2001, MPS05, MPG07} the authors presented the solutions to space-time fractional equations by means of Wright functions or Mellin-Barnes integral representations, that is Fox's functions (see also \cite{MMP10} for a review on the Mainardi-Wright function and \cite{GIL}, \cite{Mainardi}, \cite{ScalasGM}). In \cite{Nig86} the author gives a physical interpretation of $D^\beta_t u= Au$ when $A$ is the generator of a Markov process while the works \cite{Koc89, Koc90} for the first time introduced a mathematical approach. The works \cite{SaicZasl},\cite{Zasl94} proposed the fractional kinetic equation for Hamiltonian chaos. In \cite{BM01}, the time fractional problem is  studied when $A$ is an infinitely divisible generator on a finite dimensional space. Time fractional equations have been also related to (space) higher order equations (\cite{All}, \cite{AllZheng}, \cite{BonDovMaz}, \cite{DovHO}, \cite{KeyLizHO}, \cite{NaneHO}) where the solutions to higher order equations (say $n\geq 2$) are obtained by considering the time fractional equations with Caputo derivative of order $\beta =1/n$.  



In general, stochastic solutions to fractional equations can be realized through time-changes. Indeed, for a base process $X_t$ with generator $A$ we have that $X_{L_t}$ is governed by $\partial^\beta_t u= A u $ where the process $L_t$, $t>0$ is an inverse (or hitting time process) to the $\beta$-stable subordinator $H_t$, $t\geq 0$. The time fractional derivative comes from the fact that $X_{L_t}$ can be viewed as a scaling limit of continuous time random walk, where the iid jumps are separated by iid power law waiting times (see \cite{MSheff04}, \cite{MezKlaf}, \cite{RomAle94}). Results on the subordination principle for fractional evolution equations can be found in \cite{Baz2000}. Fractional equations have been therefore associated with stochastic processes in the sense that solutions to a time fractional equations can be written in terms of time-changes of base processes. The time-change we consider in case of Caputo derivative is an inverse to a stable subordinator. The Caputo derivative is related to an inverse to a stable subordinator as well as the general operators we deal with can be associated with an inverse to a general subordinator, that is, the process $X_{L_t}$ is obtained by considering a general inverse $L$. 
Such (general) time fractional operator has been recently treated in \cite{AlCV}, \cite{chen}, \cite{CKKW},  \cite{MacSchi15}, \cite{MShef08}, \cite{toaldo}. We notice that there is a consistent literature on this topic, therefore the previous references  are intended to be illustrative, and not exhaustive.

In the present paper we consider a Feller process $X$ on $(E,\mathcal{B}(E)$ with generator $(A,D(A))$ associated with regular Dirichlet form $(\mathcal{E}, \mathcal{F})$ on $L^2(E, m)$ and a subordinator  $H$ with symbol $\Phi$. We focus on the time-changed process $X^\Phi = X \circ L$ where $L$ is the inverse to the subordinator $H$.
We define the time fractional operator $\mathfrak{D}^\Phi_t$ (see Section 5) such that
\begin{align*}
\int_0^\infty e^{-\lambda t} \mathfrak{D}^\Phi_t u(t)\, dt = \Phi(\lambda) \widetilde{u}(\lambda) - \frac{\Phi(\lambda)}{\lambda} u(0), \quad \lambda > w 
\end{align*}
and we consider the fractional equation
\begin{align}
\mathfrak{D}^\Phi_t u = Au, \qquad u_0=f \in D(A).
\end{align}
The probabilistic representation of the unique solution is written in terms of the time-changed process $X^\Phi,$ that is, \begin{align}
u(t,x) = \mathbf{E}_x \left[ f(X^\Phi_t) \right] = P^\Phi_t f(x)
\end{align}
(see Theorem \ref{time-frac-THM}).
For this solution, we define the $\lambda$-potential
\begin{align}
& R^\Phi_\lambda f(x) :=  \mathbf{E}_x \left[ \int_0^\infty e^{-\lambda t} f(X^\Phi_t) dt \right], \quad \lambda> 0.
\end{align}
We consider a sequence of  processes $X^n$ on $(E^n, \mathcal{B}(E^n))$ with generators $(A^n, D(A^n))$ associated with the Dirichlet forms $(\mathcal{E}^n, \mathcal{F}^n)$, $\mathcal{F}^n = D(\mathcal{E}^n)$ on $L^2(E^n, m)$. We construct  the  sequence of  time-changed processes $X^{\Phi, n}$  related to $X^n$  and the  related  $\lambda$-potentials
\begin{align*}
R^{\Phi, n}_\lambda f(x) := \mathbf{E}_x \left[ \int_0^\infty e^{-\lambda t} f(X^{\Phi, n}_t) dt \right]  , \quad \lambda > 0
\end{align*}
associated with $(P^{\Phi, n}_t)_{t\geq 0}$. We are interested in the asymptotics of the solutions to the time fractional equations associated with generators $A^n$ of the process $X^n$ on $E^n$, $n \in \mathbb{N}$. The main goal is to obtain asymptotic results for a wide class of time-changed processes driven by fractional equations. We approach this problem by obtaining  asymptotic results in terms of M-convergences of the Dirichlet forms $(\mathcal{E}^n, \mathcal{F}^n)$ (see Theorems \ref{main1} and  \ref{main2}). More precisely,  in Theorem  \ref{main1},  we obtain that a sequence of forms $\{\mathcal{E}^n\}$  M-converges to a form $\mathcal{E}$ in  $L^2(E, m)$   if and only if the sequence  $\{ R_\lambda^{\Phi, n} : \lambda >0\}$  converges to  $ R_\lambda^\Phi$ in the strong operator topology of $L^2(E, m)$.  Moreover, by  the uniqueness of the Laplace transform, we obtain a characterization in terms of the convergence of the  sequence $P^{\Phi, n}_t$  in the strong operator topology of  $L^2(E, m)$ uniformly on every interval $0<t\leq t_1$. In Theorem \ref{main2} we obtain that a sequence of forms $\{\mathcal{E}^n\}$  M-converges to a form $\mathcal{E}$ in  $L^2(E, m)$   if and only if $X^{\Phi, n} \to X^\Phi$ in distribution as $n\to \infty$ in $\mathbb{D}$.

An useful tools is Theorem \ref{thm:lpot-L} given for $\lambda \geq 0$ by mean of which we basically handle the $\lambda$-potential of an inverse process. A further useful tool is the representation \eqref{IDuseful} involving the $\lambda$-potential of the time-changed process. 

As the theory of Dirichlet forms  provides an appropriate functional framework to the variational description of  composite media  and irregular structures like fractals, our results can be applied to several contexts  and many areas. Thus, our results provide an useful tool for studying  time fractional equations in general scenarios. As far as we know, the novelty in the present work is the connection between convergence of stochastic processes driven by time fractional equations and convergence of related forms.\\

The plan of the paper is the following. In Section 2 we recall the definition of M-convergence of forms and the characterization in terms of convergence of the associated resolvents and semigroup operators. In Section 3 we introduce Feller processes related to  regular Dirichlet form  by using \cite{chen-book, FUK-book}. In Section 4 we introduce a subordinator $H,$ its Laplace exponent $\Phi$ and its the inverse $L$ and we consider  time-changed process $X^\Phi$. In Section 5 we introduce the operator $\mathfrak{D}^\Phi_t $ and the related time fractional problems. In Section 6 we state and prove our main results. Finally, in the last section we provide some examples and applications.

\section{Convergence of forms and resolvents}

We consider the Hilbert Space $\mathcal{H}=L^2(E,m),$ where $E$ is a given separable measurable space and $m$ a $\sigma-$finite positive measure on $E$. By $(u,v)=\int_E u v \,m(dx)$ we denote the inner product of $\mathcal{H}$. 

A form $\mathcal{E}$ in $\mathcal{H}$  will be any non-negative definite symmetric bilinear form $\mathcal{E}(u,v)$ defined on a linear subspace $D(\mathcal{E})$ of $\mathcal{H}$, the domain of $\mathcal{E}$.

A form $\mathcal{E}$ is closed in $\mathcal{H}$ if its domain $D(\mathcal{E})$  is complete under the inner product $\mathcal{E}(u,v)+(u,v)$.  The closeness of a given form in $\mathcal{H}$ can be  characterized in terms its quadratic functional only: a form $\mathcal{E}$ is closed in $\mathcal{H}$ if and only if the quadratic functional $\mathcal{E}(u,u)$ is lower semi continuous on $\mathcal{H}$. 

Given a form in $\mathcal{H}$ there exist a greatest lower semi continuous functional on $\mathcal{H}$ which is a minorant of the quadratic functional $\mathcal{E}(u,u)$ associated with $\mathcal{E}$ on $\mathcal{H}$. This uniquely determined  lower semi continuous functional on $\mathcal{H}$ is also quadratic and will be denoted by $\overline{\mathcal{E}}(u,u)$.  A closed form $\overline{\mathcal{E}}(u,v)$ is then defined on the domain $D(\overline{\mathcal{E}}) = \{ u\in \mathcal{H}, \overline{\mathcal{E}}(u,u) < \infty\}$. This form, uniquely determined by the initial form $\mathcal{E}$ is the relaxation of $\mathcal{E}$ in $\mathcal{H}$ and it is called the relaxed form $\overline{\mathcal{E}}$.

Given a closed form  $\mathcal{E}$ on $\mathcal{H}$, the resolvent $\{G_\lambda, \lambda >0\}$ is uniquely defined for each $\lambda >0$ by the identity 
\begin{align*}
\mathcal{E}_\lambda(G_\lambda u, v) = (u,v), \quad \textrm{for every} \quad v \in D(\mathcal{E})
\end{align*}
where 
\begin{align*}
\mathcal{E}_\lambda(u,v)= \mathcal{E}(u,v)+\lambda (u,v) \quad \textrm{for every} \quad u,v\in D(\mathcal{E}).
\end{align*}
We consider a sequence of forms $\mathcal{E}^n$ with $D(\mathcal{E}^n)$ on $L^2(E^n, m)$ where $E^n$ is a sequence separable measurable spaces. Let $F$ be such that $E \subset F$ and $E^n \subset F$. We  recall the notion of  $\emph{$M-$convergence}$ of forms on the Hilbert space $L^2(F, m)$, introduced in \cite{MOS3} (see also \cite{MOS1}). 
\begin{definition}\label{def1}
A sequence of forms $\{\mathcal{E}^n\}$  M-converges to a form $\mathcal{E}$ in $L^2(F, m)$ if 
\begin{itemize}
\item[(a)] For every $v_n$ converging weakly to $u$ in $L^2(F, m)$
\begin{equation}
\label{bb} \liminf  \mathcal{E}^n(v_n,v_n)\geq \mathcal{E}(u,u), \quad \textrm{as} \quad n \to + \infty.
 \end{equation}
\item[(b)] For every $u \in L^2(F, m)$ there exists $u_n$ converging strongly to $u$ in  $L^2(F, m)$ such that
\begin{equation}
\label{aa} \limsup \mathcal{E}^n(u_n,u_n)\leq \mathcal{E}(u,u), \quad \textrm{as} \quad n\to + \infty.
 \end{equation}
 \end{itemize}
\end{definition}
We point out that the forms $\mathcal{E}$, $\mathcal{E}^n$ can be defined in the whole of $L^2(F, m)$ by setting 
\begin{align*}
\mathcal{E}(u,u) = +\infty \quad \forall\, u \in L^2(F, m) \setminus D(\mathcal{E}).\\
\mathcal{E}^n(u,u) = +\infty \quad \forall\, u \in L^2(F, m) \setminus D(\mathcal{E}^n).
\end{align*}

We recall the notion of $\Gamma$-convergence (a weaker convergence in the space of forms, \cite{DG}) .
\begin{definition}\label{def2}
A sequence of forms $\{\mathcal{E}^n\}$ $ \Gamma$-converges to a form $\mathcal{E}$ in $L^2(F, m)$ if 
\begin{itemize}
\item[(a)] For every $v_n$ converging strongly  to $u$ in $L^2(F, m)$
\begin{equation}
\label{bb} \liminf  \mathcal{E}^n(v_n,v_n)\geq \mathcal{E}(u,u), \quad \textrm{as} \quad n \to + \infty.
 \end{equation}
\item[(b)] For every $u \in L^2(F, m)$ there exists $u_n$ converging strongly to $u$ in  $L^2(F, m)$ such that
\begin{equation}
\label{aa} \limsup \mathcal{E}^n(u_n,u_n)\leq \mathcal{E}(u,u), \quad \textrm{as} \quad n\to + \infty.
 \end{equation}
 \end{itemize}
\end{definition}
Also in this case, the forms involved are extended in the whole space. If the sequence of forms $\{\mathcal{E}^n\}$ is asymptotically compact in $L^2(F, m)$ (that is, every sequence $u_n$ with $\liminf\{\mathcal{E}^n(u_n,u_n)+(u_n,u_n) \}< \infty$ as $n\to \infty$ has a subsequence that converges strongly in  $L^2(F, m)$), them $M$ and $\Gamma$ convergence coincide, that is, $\{\mathcal{E}^n\}$ M-converges to a form $\mathcal{E}$ in $L^2(F, m)$ if  and only if M-converges to a form $\mathcal{E}$ in $L^2(F, m)$ (see Lemma 2.3.2 in \cite{MOS1}).

The convergence of forms according the Definition \ref{def1} can be characterized in terms of convergence of the resolvent operators of the relaxed forms (see Theorem 2.4.1 in \cite{MOS1} and Theorem 3.26 in \cite{Att}).
\begin{theorem}
\label{teo1} 
A sequence of forms $\{\mathcal{E}^n\}$  M-converges to a form $\mathcal{E}$ in $L^2(F,m)$ if and only if the sequence  $\{G_\lambda^n : \lambda >0\}$ of the resolvent operators associated with the 
relaxed forms  $\{\overline{\mathcal{E}^n}\}$ converges to the resolvent operator $G_\lambda$ of the form $\mathcal{E}$ in the strong operator topology of $L^2(F,m)$.
\end{theorem}
As a consequence of Trotter-Kato characterizations of convergence of resolvent in terms of convergence of the related semigroups, (see Theorem IX 2.16 in  \cite{K}) the convergence of forms according the Definition \ref{def1} can be characterized in terms of convergence of the semigroups operators of the relaxed forms (see  Corollary 2.6.1 in \cite{MOS1}).
\begin{theorem}
\label{teo2} 
A sequence of densely defined forms $\{\mathcal{E}^n \}$  M-converges to a  densely defined form $\mathcal{E}$ in $L^2(F,m)$ if and only if for every $t>0$ the sequence  $\{T^n_t\}$ of the semigroup  operators associated with the relaxed forms  $\{\overline{\mathcal{E}^n}\}$ in $L^2(F,m)$ converges to the semigroup  operator $T_t$ associated with the form $\mathcal{E}$ in the strong operator topology of $L^2(F,m)$ uniformly on every interval $0<t\leq t_1$. 
\end{theorem}

We conclude the section by recalling that a form $\mathcal{E}$ in $\mathcal{H}$ is Markovian if the following condition is satisfied: for any $\varepsilon>0$, there exists $\eta_\varepsilon:\mathbb{R} \to [\varepsilon, 1+\varepsilon]$ with $ \eta_\varepsilon(t)=t$ for $t\in[0,1]$ and $0\leq \eta_\varepsilon(t')-\eta_\varepsilon(t)\leq t'-t$ for every $t'<t$, such that $\eta_\varepsilon\circ u\in D(\mathcal{E})$ and $\mathcal{E}(\eta_\varepsilon\circ u,\eta_\varepsilon\circ u)\leq \mathcal{E}(u,u)$ whenever $u\in D(\mathcal{E})$. A Markovian closed symmetric form  on $H$  is called a Dirichlet form and (see \cite[pag. 27]{chen-book} or \cite{FUK-book}). A Dirichlet form $(\mathcal{E}, D(\mathcal{E}))$ on $\mathcal{H}$ is said regular if:
\begin{itemize}
\item[i)] $E$ is a locally compact, separable metric space and $m$ is a Radon measure on $E$ with $supp[m]=E$,
\item[ii)] $D(\mathcal{E}) \cap C_0(E)$ is $\mathcal{E}_1$-dense in $D(\mathcal{E})$,
\item[iii)] $D(\mathcal{E}) \cap C_0(E)$ is uniformly dense in $C_0(E)$.
\end{itemize}
We denoted by $C_0(E)$ the family of continuous functions $C(E)$ on $E$ with compact support and by $C_b(E)$ the set of continuous and bounded functions on $E$. In the next section we relate Dirichlet forms with Hunt processes.

\section{Dirichlet forms and processes}

Let $E$ be a locally compact, separable metric space and $E_\partial = E \cup \{\partial \}$ be the one-point compactification of $E$ (the point $\partial$ is adjoined to $E$ as the point at infinity if $E$ is not compact and as isolated point if $E$ is compact). Denote by $\mathcal{B}(E)$ the $\sigma$-field of Borel sets in $E$ ($\mathcal{B}_\partial$ is the $\sigma$-field in $E_\partial$ generated by $\mathcal{B}$). 

Let $X=\{(X_t)_{t \geq 0}, (\mathbf{P}_x)_{x \in E}\}$ with infinitesimal generator $(A, D(A))$ be a Markov process on $(E, \mathcal{B}(E))$ with transition function $p(t, x, B)$ on $[0, \infty) \times E \times \mathcal{B}(E)$ such that (see for example \cite{BSW13, chen-book,  kazuaki}):
\begin{itemize}
\item[i)] $\forall\, f \in C_\infty(E)$, we have
\begin{align*}
\int_E p(t, x, dy)f(y) \in C_\infty(E),
\end{align*}
where $C_\infty(E)$ is the set of continuous functions $C(E)$ on $E$ such that $u(x)\to 0$ as $x \to \partial$,
\item[ii)] $\forall\, \epsilon>0$, for the $\epsilon$-neighbourhood $U_\epsilon(x) = \{ y \in E\,:\, d(x,y) < \epsilon \}$, we have
\begin{align*}
\lim_{t \downarrow 0} \sup_{x \in E} p(t, x, E \setminus U_\epsilon(x) )= 0 ,
\end{align*}
\item[iii)] $\forall\, T>0$ and each compact $K \subset E$, we have
\begin{align*}
\lim_{x \to \partial} \sup_{0 \leq t \leq T } p(t, x, K) = 0.
\end{align*}
\end{itemize}
Throughout the paper, we use the fact that a function $f$ on $E$ can be extended to $E_\partial$ by setting $f(\partial) = 0$. The point $\partial$ is the cemetery point for $X$. Moreover, we write
\begin{align*}
P_t f(x) = \mathbf{E}_x[f(X_t)]
\end{align*}
where by $\mathbf{E}_x$ we denote the mean value with respect to $\mathbf{P}_x$ (the process starts from $x \in E$). From the Riesz representation theorem, the operator
\begin{equation}
P_tf(x) := \int_E p(t,x,dy) f(y), \quad f \in C_\infty(E)
\end{equation}
is uniquely defined and condition $i)$ says that $C_\infty(E)$ is an invariant subset of $C(E)$ for $P_t$. Since $p(t, x, \cdot)$ is a $C_\infty$-transition function (or Feller transition function), $(P_t)_{t\geq 0}$ is a non-negative and contraction semigroup on $C_\infty(E)$. This together with condition $ii)$ (that is, $X$ is uniformly stochastically continuous on $E$) and condition $iii)$ say that $P_t$ is strongly continuous in $t$ on $C_\infty(E)$. Then, $X$ is a strong Markov process, right-continuous with no discontinuity other that jumps (Feller process, see for example \cite[pag. 413]{chen-book}, \cite[Chapter 1]{BSW13}, \cite[pag. 432]{kazuaki}). For a Feller process with transition function $p(t, \cdot)$, there exists an Hunt process with transition function given by $p(t, \cdot)$ (\cite[Theorem A.2.2]{FUK-book}). Two processes on the same state space $(E, \mathcal{B}(E))$ are said to be equivalent if they have the same transition function (\cite[Definition 4.1]{BluGet68}). Let $m$ be a Radon measure on $E$ with $supp[m]=E$ such that $(P_t u, v) = (u, P_tv)$, that is, $P_t$ is $m$-symmetric. 
The transition function of an $m$-symmetric Hunt process uniquely determines a strongly continuous Markovian semigroup $T_t$ on $L^2(E, m)$ and the Dirichlet form $\mathcal{E}$ on $L^2(E, m)$ (see \cite[pag. 141]{FUK-book}). We recall that there is one to one correspondence between the family of closed symmetric forms on $\mathcal{H}$ and the family of non-positive definite self-adjoint operators $A$ on $\mathcal{H}$ where the correspondence is determined by $D(\mathcal{E})=D(\sqrt{-A})$ and $\mathcal{E}(u,v)=(\sqrt{-A} u, \sqrt{-A}v)$ (see \cite[Theorem 1.3.1]{FUK-book}).  Then, $X$ is equivalent to an $m$-symmetric Hunt process (see \cite[Theorem A.1.43]{chen-book} obtained by Chapter 1 of \cite[Theorem 1.9.4]{BluGet68}) on $(E, \mathcal{B}(E))$ whose Dirichlet form $(\mathcal{E}, D(\mathcal{E}))$ is on $L^2(E, m)$ (\cite[Theorem 1.5.1]{chen-book}). In particular, $P_t u(x) = T_tu(x)$, $x \in E \setminus N$ with $m(N)=0$ (see \cite[Proposition 3.1.9]{chen-book}). If the transition function also satisfies:
\begin{itemize}
\item[iv)] $\forall\, \epsilon>0$ and each compact $K \subset E$, we have
\begin{align*}
\lim_{t\downarrow 0} \frac{1}{t}\sup_{x \in K} p(t, x, E \setminus U_\epsilon(x)) = 0,
\end{align*}
\end{itemize}
then, $p(t,x,y)$ is a transition function of some continuous strong Markov process. In this case we say that $X$ is a diffusion (Feller diffusion) and the corresponding Dirichlet form is local (\cite[Theorem 4.5.1]{FUK-book}).

Throughout, we assume that the form $(\mathcal{E}, D(\mathcal{E}))$ on $L^2(E, m)$ of the associated Hunt process is regular (not necessarily equivalent to a Feller diffusion). Such a condition is not restrictive because for any given Dirichlet form $\mathcal{E}$, there exists uniquely an $m$-symmetric Hunt process whose Dirichlet form is $\mathcal{E}$ (\cite[pag. 143]{FUK-book}).

\section{Time-changes of processes}
Let $H=\{H_t\, , \, t \geq 0\}$ be a subordinator (see \cite{Bertoin} for a detailed discussion). Then, $H$ can be characterized by the Laplace exponent $\Phi$, that is, $$\mathbf{E}_0[\exp( - \lambda H_t)] = \exp(- t \Phi(\lambda)),$$
$\lambda \geq 0$. Moreover, if $\Phi$ is the Laplace exponent of a subordinator, then there exists a unique pair $(\mathtt{k}, \mathtt{d})$ of non-negative real numbers and a unique measure $\Pi$ on $(0, \infty)$ with $\int (1 \wedge z) \Pi(dz) < \infty$, such that for every $\lambda \geq 0$
\begin{equation}
\Phi(\lambda) = \mathtt{k} + \mathtt{d} \lambda + \int_0^\infty \left( 1 - e^{ - \lambda z} \right) \Pi(dz). \label{LevKinFormula}
\end{equation} 
The L\'evy-Khintchine representation in formula \eqref{LevKinFormula} is written in terms of the killing rate $\mathtt{k}=\Phi(0)$ and the drift coefficient $\mathtt{d}= \lim_{\lambda\to \infty} \Phi(\lambda)/\lambda$ where
\begin{align}
\label{tailSymb}
\frac{\Phi(\lambda)}{\lambda} = \mathtt{d} +  \int_0^\infty e^{-\lambda z} \overline{\Pi}(z)dz, \qquad \overline{\Pi}(z) = \mathtt{k} + \Pi((z, \infty))
\end{align}
and $\overline{\Pi}$ is the so called \emph{tail of the L\'evy measure}. We recall that $\Phi$ is a Bernstein function uniquely given by \eqref{LevKinFormula}, then $\Phi$ is a non-negative, non-decreasing and continuous function.  For details, see \cite{BerBook}. 

We define the inverse process $L =\{L_t\, , \, t \geq 0\}$ to a subordinator as $$L_t := \inf\{ s \geq 0\,:\, H_s \notin (0, t)\}.$$ We do not consider step-processes with $\Pi((0,\infty))<\infty$,  we focus only on strictly increasing subordinators with infinite measures (then $L$ turns out to be a continuous process). By definition of inverse process, we can write
\begin{align}
\label{relationPHL}
\mathbf{P}_0(L_t < s) = \mathbf{P}_0(H_s>t)
\end{align} 
with $H_0=0$ and $L_0=0$. The subordinator $H$ and the inverse $L$ can be respectively regarded as an hitting time and a local time of some Markov process (\cite{Bertoin}). Let us consider the independent processes $(X, L)$ introduced before. As usual we denote by $\zeta$ the lifetime of $X$.  If  $P_t$ is conservative the process $X$ has infinite lifetime and $p(t,x,E)=1$. In the following discussion we do not distinguishes between $E$ and $E_\partial$ for the non conservative case ($\partial$ is the absorbing set, for instance), we focus only on $\zeta$. Then we consider the base process $X=\{X_t,\; t < \zeta\}$. We focus now on the time-changed process $X^\Phi = \{ X^\Phi_t:=X_{L_t}\, ,\, L_t < \zeta\}$ associated with
\begin{align}
\label{semigPhi}
P^\Phi_t f(x) := \int_0^\infty P_s f(x) \mathbf{P}_0(L_t \in ds) = \mathbf{E}_x[f(X_{L_t}), \, L_t < \zeta]. 
\end{align}
Observe that $P_t^\Phi$ is not a semigroup (for $\Phi \neq id$, the identity map). Since $L$ is continuous we always have
\begin{align}
\label{semigPhi-L-continuity}
\mathbf{E}_x[f(X^\Phi_t), t < \zeta^\Phi] = \mathbf{E}_x[f(X^\Phi_t), L_t < \zeta]
\end{align}
where $\zeta^\Phi$ is the lifetime of $X^\Phi$. This also means that
\begin{align}
\label{semigPhi-1}
P^\Phi_t \mathbf{1}_E = \mathbf{P}_x(t < \zeta^\Phi) = \mathbf{P}_x(L_t < \zeta) = \int_0^\infty \mathbf{P}_x(s < \zeta) \mathbf{P}_0(L_t \in ds)
\end{align}
where the last identity can be obtained from \eqref{semigPhi} or by considering that $\zeta$, $L$ are independent.

\section{PDEs connection}
Let $M>0$ and $w\geq 0$. Let $\mathcal{M}_w$ be the set of (piecewise) continuous function on $[0, \infty)$ of exponential order $w$ such that $|u(t)| \leq M e^{wt}$. Denote by $\widetilde{u}$ the Laplace transform of $u$. Then, we define the operator $\mathfrak{D}^\Phi_t : \mathcal{M}_w \mapsto \mathcal{M}_w$ such that
\begin{align*}
\int_0^\infty e^{-\lambda t} \mathfrak{D}^\Phi_t u(t)\, dt = \Phi(\lambda) \widetilde{u}(\lambda) - \frac{\Phi(\lambda)}{\lambda} u(0), \quad \lambda > w
\end{align*}
where $\Phi$ is given in \eqref{LevKinFormula} with $\mathtt{d}=0$ and $\mathtt{k}=0$. Since $u$ is exponentially bounded, the integral $\widetilde{u}$ is absolutely convergent for $\lambda>w$.  By Lerch's theorem the inverse Laplace transforms $u$ and $\mathfrak{D}^\Phi_tu$ are uniquely defined. We note that
\begin{align}
\label{PhiConv}
\Phi(\lambda) \widetilde{u}(\lambda) - \frac{\Phi(\lambda)}{\lambda} u(0) = & \left( \lambda \widetilde{u}(\lambda) - u(0) \right) \frac{\Phi(\lambda)}{\lambda}
\end{align}
and thus, $\mathfrak{D}^\Phi_t$ can be written as a convolution involving the ordinary derivative and the inverse transform of \eqref{tailSymb} iff $u \in \mathcal{M}_w \cap C([0, \infty), \mathbb{R}_+)$ and $u^\prime \in \mathcal{M}_w$. By Young's inequality we also observe that
\begin{align}
\label{YoungSymb}
\int_0^\infty |\mathfrak{D}^\Phi_t u |^p dt \leq \left( \int_0^\infty |u^\prime |^p dt \right) \left( \lim_{\lambda \downarrow 0} \frac{\Phi(\lambda)}{\lambda} \right)^p, \qquad p \in [1, \infty)
\end{align}
where $\lim_{\lambda \downarrow 0} \Phi(\lambda) /\lambda$ is finite only in some cases, for example:
\begin{itemize}
\item inverse Gaussian subordinator with $\Phi(\lambda) = \sigma^{-2} \left(\sqrt{2\lambda \sigma^2 +\mu^2} - \mu \right)$ with $\sigma \neq 0$;
\item gamma subordinator with $\Phi(\lambda)= a \ln (1+ \lambda/b)$  with $ab>0$;
\item generalized stable subordinator with $\Phi(\lambda) = (\lambda + \gamma)^\alpha - \gamma^\alpha$ with $\gamma>0$ and $\alpha \in (0,1)$.
\end{itemize}
We notice that when $\Phi(\lambda)=\lambda$ (that is we deal with the ordinary derivative $D_t$) we have that $H_t = t$ and $L_t=t$ a.s. and in \eqref{YoungSymb} the equality holds. 
Explicit representations of $\mathfrak{D}^\Phi_t$ in terms of the tails of a L\'evy measure have been given in the recent works \cite{chen, toaldo}. 

In the present work  we consider the equation
\begin{align}
\label{time-frac-problem}
\mathfrak{D}^\Phi_t u = Au, \qquad u_0=f \in D(A)
\end{align}
whose probabilistic representation of the solution is written in terms of the time-changed process $X^\Phi$ obtained by the base process $X$ with generator $(A, D(A))$ and $L$ introduced above.  In particular, the process $X^\Phi$ can be considered in order to study the solution to \eqref{time-frac-problem}, that is (see formula \eqref{semigPhi})
\begin{align}
\label{sol-time-frac-problem}
u(t,x) = P^\Phi_t f(x).
\end{align}

First we provide the following useful result for $\lambda \geq 0$ which slightly generalizes the result in \cite{chen}.
\begin{theorem}
\label{thm:lpot-L}
Let $f \in \mathcal{M}_w$. Then,
\begin{equation}
\mathbf{E}_0\left[ \int_0^\infty e^{-\lambda t} f(L_t) dt \right] =\frac{\Phi(\lambda)}{\lambda} \mathbf{E}_0 \left[ \int_0^\infty e^{- \lambda H_t} f(t)\, dt  \right], \quad \lambda> \Phi^{-1}(w).
\end{equation}
\end{theorem}
\begin{proof}
From \eqref{relationPHL} we get that
\begin{align*}
\mathbf{E}_0\left[ \int_0^\infty e^{-\lambda t} f(L_t) dt \right]= & \int_0^\infty \int_0^\infty e^{-\lambda t} f(s) \mathbf{P}_0(L_t \in ds) \, dt \\
= & \int_0^\infty \int_0^\infty e^{-\lambda t}\left( - \frac{d}{d s} \mathbf{P}_0(H_s \leq t) \right)  f(s)\, dt \,ds\\ 
= &  \int_0^\infty \left( - \frac{d}{d s} \frac{e^{- s \Phi(\lambda)}}{\lambda} \right) f(s)\, ds\\
= & \frac{\Phi(\lambda)}{\lambda} \int_0^\infty e^{-s\Phi(\lambda)} f(s)\, ds, \quad \Phi(\lambda)> w.
\end{align*}
Since $\Phi$ is non-decreasing, $\Phi^{-1}$ is non-decreasing and the proof is completed.
\end{proof}

We now study the problem \eqref{time-frac-problem} by first considering the problem $\partial_t u = A u$ with initial datum $f \in D(A)$. The solution $u$ is unique and has the probabilistic representation $P_tf(x)=\mathbf{E}_x \left[ f(X_t)\right]$ with $\lambda$-potential 
\begin{align}
& R_\lambda f(x) :=  \mathbf{E}_x \left[ \int_0^\infty e^{-\lambda t} f(X_t) dt  \right], \quad \lambda>0.
\end{align}
We recall that, for each $\lambda>0$, $G_\lambda$ introduced in Theorem \ref{teo1} is a quasi continuous version of $R_\lambda$ (\cite[Proposition 3.1.9]{chen-book}). For the solution to \eqref{time-frac-problem} with $f \in D(A)$ we define the $\lambda$-potential
\begin{align}
& R^\Phi_\lambda f(x) :=  \mathbf{E}_x \left[ \int_0^\infty e^{-\lambda t} f(X^\Phi_t) dt \right], \quad \lambda> 0
\end{align}
and obtain the following result already given in \cite{chen, toaldo} in alternative forms, also due to the definition of $\mathfrak{D}^\Phi_t$. In \cite{chen} the author considers a strong Markov process associated with a uniformly bounded strongly continuous semigroup in some Banach space, the fractional equation governing the time-changed process involves a time fractional operator of Riemann-Liouville type. In \cite{toaldo}, the author considers $C_0$-semigroups and the fractional operator in time of Caputo type. Previous works focus on pseudo-differential operators (\cite{MShef08}, for example) or integro-differential operators (\cite{MacSchi15}). We follow a different approach based on the simple relation \eqref{IDuseful} below which turns out to play a key role also in the proof of the main results of our work (see formula \eqref{IDusefuln} below).   

\begin{theorem}
\label{time-frac-THM}
The function \eqref{sol-time-frac-problem} is the unique strong solution in $L^2(E, m)$ to \eqref{time-frac-problem} in the sense that:
\begin{enumerate}
\item $\varphi: t \mapsto u(t, \cdot)$ is such that $\varphi \in C([0, \infty), \mathbb{R}_+)$ and $\varphi^\prime  \in \mathcal{M}_0$,
\item $\vartheta : x \mapsto u(\cdot, x)$ is such that $\vartheta, A\vartheta \in D(A)$,
\item $\forall\, t > 0$, $\mathfrak{D}^\Phi_t u(t,x) = Au(t,x)$ holds $m$-a.e in $E$,
\item $\forall\, x \in E$, $u(t,x) \to f(x)$ as $t \downarrow 0$.
\end{enumerate}
\end{theorem}
\begin{proof}
From Theorem \ref{thm:lpot-L}, the $\lambda$-potential of the time-changed process is given by
\begin{align*}
R^\Phi_\lambda f(x) = \frac{\Phi(\lambda)}{\lambda} \int_0^\infty e^{-s\Phi(\lambda)} P_s f(x)\, ds =  \frac{\Phi(\lambda)}{\lambda} R_{\Phi(\lambda)} f(x)
\end{align*}
from which we obtain the useful identity
\begin{equation}
\label{IDuseful}
\lambda R^\Phi_\lambda f = \Phi(\lambda) R_{\Phi(\lambda)} f .
\end{equation}
Since $P_t$ is strongly continuous, we have :
\begin{align*}
f= \lim_{\lambda \to \infty} \Phi(\lambda) R_{\Phi(\lambda)} f = \lim_{\lambda \to \infty} \lambda R^\Phi_\lambda f \quad \textrm{that is} \quad P^\Phi_t f \to f \textrm{ as } t \downarrow 0 
\end{align*}
and, for $f \in D(A)$, the mapping $ [0, \infty ) \ni t \mapsto P_tf \in C_\infty(E)$ is differentiable, $\frac{d}{d t}P_tf = AP_t f = P_t Af$. Thus, $A R_\lambda f = R_\lambda A f$ and we obtain the identity 
\begin{align*}
A R^\Phi_\lambda f(x) = \frac{\Phi(\lambda)}{\lambda}  \mathbf{E}_x \left[ \int_0^\infty e^{-t \Phi(\lambda)} Af(X_t) dt  \right].
\end{align*}
Notice that, for $\lambda > 0$, we have that $R^\Phi_\lambda f, AR^\Phi_\lambda f \in D(A)$, 
\begin{align*}
\sup_{x \in E} |R^\Phi_\lambda f(x)| \leq \sup_{x \in E} \frac{\Phi(\lambda)}{\lambda} \int_0^\infty e^{- t \Phi(\lambda)} \big| \mathbf{E}_x[f(X_t)] \big| dt \leq \frac{M}{\lambda}  \sup_{x \in E} |f(x)|,
\end{align*}
and, let $g=Af$,
\begin{align*}
\sup_{x \in E} |A R^\Phi_\lambda f(x)| \leq \sup_{x \in E} \frac{\Phi(\lambda)}{\lambda} \int_0^\infty e^{- t \Phi(\lambda)} \big| \mathbf{E}_x[g(X_t)] \big| dt.
\end{align*}
Under the assumptions above on $A$ (and therefore on $P_t$), from the Dynkin's formula,  the process  
\begin{align*}
X_t^f = f(X_t) - f(X_0) - \int_0^t Af(X_s) ds
\end{align*}
with $f \in D(A)$ is a martingale under $(\mathbf{P}_{x})_{x \in E}$ (the martingale problem is uniquely solvable, see \cite{CASTEREN} for instance). After simple manipulation, we obtain that
\begin{align*}
R_\lambda Af(x) = & \mathbf{E}_x \left[ \int_0^\infty e^{- \lambda t} Af(X_t) dt  \right]\\ 
= & \lambda \mathbf{E}_x \left[ \int_0^\infty e^{- \lambda t} \int_0^t Af(X_s) ds \, dt  \right]\\
= &  \lambda \mathbf{E}_x \left[ \int_0^\infty e^{- \lambda t} \left( f(X_t) - f(X_0) \right) dt  \right] =   \lambda R_{\lambda} f(x) - f(x).
\end{align*}
From \eqref{IDuseful}, we can write
\begin{align*}
\mathbf{E}_x \left[ \int_0^\infty e^{-t \Phi(\lambda)} Af(X_t) dt  \right] = \lambda R^{\Phi}_\lambda f(x) - f(x)
\end{align*}
and we get that, $\forall\, x \in E \setminus N$ (with $m(N)=0$)
\begin{align*}
A R^{\Phi}_\lambda f(x) = \Phi(\lambda) R^\Phi_\lambda f(x) - \frac{\Phi(\lambda)}{\lambda} f(x) = \int_0^\infty e^{-\lambda t} \mathfrak{D}^\Phi_t u(t,x) dt 
\end{align*}
and we find a solution in $L^2(E, m)$.

We prove uniqueness of the solution by considering 
\begin{align*}
u^*(t, x) = u(t,x) - \mathbf{E}_x \left[ f(X^\Phi_t) \right]
\end{align*} 
where $u$ is a solution to \eqref{time-frac-problem} with
\begin{align*}
R^*_\lambda f(x) = \widetilde{u}(\lambda, x) - R^\Phi_\lambda f(x).
\end{align*}
Then, $u^*$ solves \eqref{time-frac-problem} with $u^*(0,x) =0$. 
Let us assume that $v_1$ and $v_2$ are two different solutions to $\partial_t v = Av$ with initial datum $f$. The operator $\mathfrak{D}^\Phi_t$ is uniquely defined by its Laplace symbol. Then, from the uniqueness of the Laplace transform and Theorem \ref{thm:lpot-L}, we can write from \eqref{semigPhi},
\begin{align*}
R^*_\lambda f(x) = \frac{\Phi(\lambda)}{\lambda} \int_0^\infty e^{-t \Phi(\lambda)} \left(v_1(t,x) - v_2(t, x) \right) dt.
\end{align*} 
From the uniqueness on $L^2(E,m)$ of the solution $v$ we have that
%
\begin{align*}
R^*_\lambda f(x) = \int_0^\infty e^{-\lambda t} u^*(t, x) dt = 0 \quad \textrm{which implies} \quad u^*(t, x)=0\; \forall\, t, \; m\textrm{-a.e.}
\end{align*}
and this concludes the proof.
\end{proof}

\begin{remark}
In \cite{chen} the author proves existence and uniqueness of strong solutions to general time fractional equations with initial data $f \in D(A)$. In \cite{CKKW} the authors establish existence and uniqueness for weak solutions and initial data $f \in L^2$.

\end{remark}

\begin{remark}
As also pointed out recently in \cite{chen}, $X^\Phi$ can have infinite lifetime. Indeed, we have that $P^\Phi_t \mathbf{1}_E(x) = \mathbf{E}_x[\mathbf{1}_{(t < \zeta^\Phi)}]$ and then 
\begin{align}
\label{meanExample}
\mathbf{E}_x[\zeta^\Phi] = \lim_{\lambda \downarrow 0} R^\Phi_\lambda \mathbf{1}_{E}(x) = \left( \lim_{\lambda \downarrow 0} \frac{\Phi(\lambda)}{\lambda} \right) \mathbf{E}_{x}[\zeta] = \Phi^\prime(0) \mathbf{E}_{x}[\zeta]
\end{align}
with $\Phi^\prime(0)=\mathbf{E}_0[H_1]>0$. 

Notice that the process $X^\Phi$ could have a finite lifetime and an infinite mean lifetime. If $P_t$ is conservative, $\mathbf{P}_x(\zeta > t) = 1$ for all $t\geq 0$ and the mean lifetime if obviously infinite. Indeed, if $\mathbf{P}_x(X_t \in E) = 1$, then $R^\Phi_\lambda \mathbf{1}_{E}(x) = 1/\lambda$. On the other hand, the fact that $\mathbf{P}_x(X_t \in E_\partial)=1$ says that $\mathbf{P}_x(\zeta>t) \leq  1$. 

Let $\zeta$ be exponentially distributed with some parameter $c>0$ independently from the starting point $x\in E$. Then, $\mathbf{E}[\zeta]=1/c < \infty$.  We exactly have that
\begin{align*}
\mathbf{P}_x(\zeta^\Phi > t) = \mathbf{E}_0[e^{-cL_t}] \quad \textrm{and} \quad \lim_{x \to \partial} \mathbf{P}_x(\zeta^\Phi > t)=0.
\end{align*}
By considering $f(s)=e^{-cs}$ in Theorem \ref{thm:lpot-L}, we obtain
\begin{align*}
R^\Phi_\lambda \mathbf{1}_{E}(x)=\int_0^\infty e^{-\lambda t} \mathbf{E}_0[e^{-cL_t}] dt = \frac{1}{\lambda} \frac{\Phi(\lambda)}{c + \Phi(\lambda)}, \quad \lambda > \Phi^{-1}(c)
\end{align*}
whose asymptotic behaviour agrees with \eqref{meanExample}.
\end{remark}

\section{Main results}

Let $X$ with generator $(A, D(A))$ be the process on $(E,\mathcal{B}(E)$ associated with the Dirichlet form $(\mathcal{E}, \mathcal{F})$, $\mathcal{F}=D(\mathcal{E})$ on $L^2(E, m)$. Let $X^n$ and $X^{\Phi, n}$ be sequences of processes as introduced in the previous sections. The processes $X^n$ on $(E^n, \mathcal{B}(E^n))$ with generator $(A^n, D(A^n))$ are associated with the Dirichlet forms $(\mathcal{E}^n, \mathcal{F}^n)$, $\mathcal{F}^n = D(\mathcal{E}^n)$ on $L^2(E^n, m)$. Let us write the sequences of $\lambda$-potentials
\begin{align*}
R^{n}_\lambda f(x) := \mathbf{E}_x \left[ \int_0^\infty e^{-\lambda t} f(X^{n}_t) dt \right], \quad \lambda > 0
\end{align*}
associated with $(P^n_t)_{t \geq 0}$ and
\begin{align*}
R^{\Phi, n}_\lambda f(x) := \mathbf{E}_x \left[ \int_0^\infty e^{-\lambda t} f(X^{\Phi, n}_t) dt \right]  , \quad \lambda > 0
\end{align*}
associated with $(P^{\Phi, n}_t)_{t\geq 0}$. From \eqref{IDuseful}, we recover the identity
\begin{align}
\label{IDusefuln}
\lambda R^{\Phi, n}_\lambda = \Phi(\lambda) R^n_{\Phi(\lambda)}, \quad n \in \mathbb{N}.
\end{align}

Denote by $T^{\Phi,n}_t$ and $G^{\Phi,n}_\lambda$ the quasi continuous versions of $P^{\Phi,n}_t$ and $R^{\Phi,n}_\lambda$ respectively. Denote by $T^\Phi_t$ and $G^\Phi_\lambda$ the quasi continuous versions of $P^\Phi_t$ and $R^\Phi_\lambda$ respectively. The discussion below is concerned with the limit object $X^\Phi = X \circ L$ on $(E, \mathcal{B}(E))$.

\begin{theorem}\label{main1}
A sequence of forms $\{\mathcal{E}^n\}$  M-converges to a form $\mathcal{E}$ in  $L^2(F, m)$   if and only if the sequence  $\{ G_\lambda^{\Phi, n} : \lambda >0\}$  converges to  $ G_\lambda^\Phi$ in the strong operator topology of $L^2(F, m)$.
A sequence of densely defined forms $\{\mathcal{E}^n \}$  M-converges  to a form $\mathcal{E}$ in  $L^2(F, m)$ if and only if for every $t>0$ the sequence $T^{\Phi, n}_t f_n$ converges to $T^\Phi_t $  in the strong operator topology of  $L^2(F, m)$ uniformly on every interval $0<t\leq t_1$. 
\end{theorem}
\begin{proof}
First we consider the sequence $G^n_\lambda$. From Theorem \ref{teo1},  we have convergence of the forms ${\mathcal{E}^n}$ if and only if we  have convergence of the associated resolvents $G^n_\lambda $ and the corresponding semigroups $T^n_t.$    Since $G^n_\lambda$ is a quasi continuous version of $R^n_\lambda$ which is related to $R^{\Phi, n}_\lambda$ by formula \eqref{IDusefuln}, we can easily obtain  by formula \eqref{IDuseful}, that $G^{\Phi, n}_\lambda\to G^\Phi_\lambda $  in the strong operator topology of $L^2(F, m)$.
By  the uniqueness of the Laplace transform we obtain the last characterization.
\end{proof}

Convergence of semigroups $P^n_t$ (the quasi continuous version of $T^n_t$) implies convergence of the finite dimensional distributions, indeed from the Markov property we have that
\begin{align*}
\mathbf{E}_x[\mathbf{1}_{B_1}(X^n_{t_1}),\mathbf{1}_{B_2}(X^n_{t_2}),\ldots, \mathbf{1}_{B_k}(X^n_{t_k})] = P^n_{t_1-t_0}\mathbf{1}_{B_1} P^n_{t_2-t_1}\mathbf{1}_{B_2} \cdots P^n_{t_k - t_{k-1}}\mathbf{1}_{B_k}(x).
\end{align*}

Let $\mathbb{D}$ be the set of continuous functions from $[0, \infty)$ to $E_\partial$ which are right continuous on $[0, \infty)$ with left limits on $(0, \infty)$. Let $\mathbb{D}_0$ the set of non-decreasing continuous function from $[0, \infty) $ to $[0, \infty)$.

\begin{theorem}
\label{teoK}
(Kurtz, \cite{KurtzRTC}. Random time change theorem). Suppose that $X^n$, $X$ are in $\mathbb{D}$ and $L^n$, $L$ are in $\mathbb{D}_0$. If $(X^n, L^n)$ converges to $(X, L)$ in distribution as $n\to \infty$, then $X^n \circ L^n$ converges to $X \circ L$ in distribution as $n \to \infty$.
\end{theorem}
\begin{proof}
The proof follows from part b) of Theorem 1.1 and part a) of Lemma 2.3 in \cite{KurtzRTC}.
\end{proof}

The sequence $X^n$ on $F$ is right-continuous with no discontinuity other than jumps and converges to a process $X$ on $F$ with generator $(A,D(A))$ associated with $(\mathcal{E}, \mathcal{F})$ on $L^2(F, m)$. Indeed, $(\mathcal{E}, \mathcal{F})$ is a regular Dirichlet form on $L^2(F, m)$, then there exists $X$ which is an Hunt process with an $m$-symmetric transition function so that $(\mathcal{E}, \mathcal{F})$ is the Dirichlet form of the transition function of $X$ (\cite[Theorem 1.5.1]{chen-book}).

\begin{theorem}\label{main2}
A sequence of forms $\{\mathcal{E}^n\}$  M-converges to a form $\mathcal{E}$ in  $L^2(F, m)$   if and only if $X^{\Phi, n} \to X^\Phi$ in distribution as $n\to \infty$ in $\mathbb{D}$.
\end{theorem}

\begin{proof}
From the $M$-convergence of the forms we have that $P^n_t f \to P_t f$ strongly in $L^2(E, m)$, see Theorem \ref{teo2}. We use Theorem 17.25 (Trotter, Sova, Kurtz, Mackevi\v{c}ius) in \cite{Kal} by means of which we have that strong convergence of semigroups (Feller semigroups) is equivalent to weak convergence of measures if $X^n_0 \to X_0$ in distribution in $E$. So, we obtain that $X^n \stackrel{d}{\to} X$ in $\mathbb{D}$. By Theorem \ref{teoK} and the convergence in distribution of $(X^n, L)$, we conclude that $X^{\Phi, n} \stackrel{d}{\to} X^{\Phi}$ in $\mathbb{D}$.

If $X^{\Phi, n} \stackrel{d}{\to} X^{\Phi}$, then we have that, $\forall\, x \in E$, $\forall\, t>0$
\begin{align*}
\mathbf{E}_x[f(X^{\Phi, n}_t)] \to \mathbf{E}_x[f(X^{\Phi}_t)] \quad \forall\, f \in C_b
\end{align*}
and therefore, $\forall\, f \in D(A)$. Thus, $P^{\Phi, n}_t \to P^\Phi_t$ in the strong operator topology of  $L^2(E, m)$ uniformly on every interval $0<t\leq t_1$. So, by Theorem \ref{main1} we obtain the convergence of $\mathcal{E}^n$. 
\end{proof}

\begin{remark}
We observe that our result can be extended to fractional operators characterized by a sequence $\Phi_n$ converging to $\Phi$. Indeed, Theorem \ref{teoK} holds  by considering a sequence $L^n$ associated with $\Phi_n$ (in this direction some results have been proved in \cite{ChenSong} by considering only subordination of symmetric Markov processes). We note that we can consider varying Hilbert spaces by using generalized Mosco convergence (\cite{Kol, KuShi}).
\end{remark}

\begin{remark}
Let us consider the process $Z_t$, $t\geq 0$, governed by the fractional problem
\begin{align*}
\mathfrak{D}^\Phi_t u = - \Psi(-A) u, \quad u_0 \in D(-\Psi(-A)) \subset D(A)
\end{align*}
where $\Psi$ has a representation \eqref{LevKinFormula} with $\mathtt{k}=\mathtt{d}=0$ for some L\'evy measure. 

Let $L$ be the inverse of $H$ with symbol $\Phi$ as in the previous sections. Let $L^*$ be the inverse to $H^*$ with symbol $\Psi$. Then $Z$ can be represented through subordination (by $H^*$) and time-change (by $L^*$) of $X$ with generator $A$. We note that
\begin{align*}
\mathbf{E}_x \left[ \int_0^\infty e^{-\lambda t} f(Z_t) dt \right] := R^{\Psi(\Phi)}_\lambda f(x)
\end{align*}
can be written as
\begin{align*}
R^{\Psi(\Phi)}_\lambda f(x) =\frac{\Phi(\lambda)}{\lambda} \left( \frac{\Psi(\Phi(\lambda))}{\Phi(\lambda)} \right)^2 R_{\Psi(\Phi(\lambda))} f(x)
\end{align*}
which generalizes \eqref{IDuseful}.
\end{remark}

\section{Examples and applications}

We point out that Theorems \ref{main1} and  \ref{main2}   allows us to obtain  asymptotic  results for fractional equations via  M-convergences of the corresponding energy forms or via $\Gamma$-convergence if the forms are asymptotically compact (see Section 2).

Now we focus our attention on some asymptotic results obtained previously by the authors just to give to the readers few simple examples.

\subsection{Asymptotic for skew diffusions on regular domains}
We consider the results in \cite{CapDovAsy} which can be associated with the sequence $X^n$ and obtain results as in the previous sections for $X^{\Phi, n}$ driven by the time fractional equation \eqref{time-frac-problem}. 

Let $\Omega_q$ with radius $q=l,\ell,r$ where $r=\ell+\epsilon$ and $\epsilon>0$, be the balls centred at the same point and such that $\Omega_l \subset \Omega_\ell \subset \Omega_r$.  Let $B^1, B^2$ be two independent Brownian motions and define the process (see the generator \eqref{A-form} below)
\begin{align*}
B^{(\alpha, \eta)}_t := \left\lbrace
\begin{array}{ll}
\displaystyle B^1_t & \textrm{on } \Omega_\ell \setminus \overline{\Omega_l} \\
\displaystyle B^2_{\eta t} & \textrm{on } \Sigma_\epsilon = \Omega_r \setminus \overline{\Omega_\ell}
\end{array}
\right. 
\end{align*}
with skew condition on $\partial \Omega_\ell$:
\begin{align*}
\forall\, x \in \partial \Omega_\ell, \qquad \mathbb{P}_x(B^{(\alpha, \eta)}_t \in \Omega_\ell \setminus \overline{\Omega_l}) =1-\alpha \quad \textrm{and} \quad \mathbb{P}_x(B^{(\alpha, \eta)}_t \in \Sigma_\epsilon) = \alpha.
\end{align*}
Moreover, we require that $B^{(\alpha, \eta)}$ is killed on $\partial \Omega_l$ and $\partial \Omega_r$. Since we have different variances depending on $\eta >0$, we refer to the process $B^{(\alpha, \eta)}$ as a modified process. Obviously, $\alpha \in (0,1)$ is the skewness parameter and $B^{(\alpha, \eta)}$ is called (modified) skew process. We write $\alpha=\alpha_n$, $\eta=\eta_n$, $\epsilon=\epsilon_n$ by underling the dependence from $n$ ($\alpha_n\to 0$, $\eta_n\to 0$ and $\epsilon_n \to 0$ as $n \to \infty$) and we consider the collapsing domain $\Sigma_\epsilon$ (that is, with vanishing thickness $\epsilon$). Our aim is to study  a killed diffusion on $\Omega_r \setminus \overline{\Omega_l}$ with skew condition on $\partial \Omega_\ell$ and different behaviour in $\Sigma_\epsilon=\Omega_r \setminus \overline{\Omega_\ell}$ and $\Omega_\ell \setminus \overline{\Omega_l}$  under the assumption \begin{align}
\lim_{n \to \infty} \frac{\alpha \epsilon}{\eta} =0.
\label{main-condition}
\end{align}
The classical case $\alpha=\eta$ has been extensively investigated in literature (see for example \cite{AB,BCF,ButtazzoDMasoMosco} and the references therein): in this case,  the condition \eqref{main-condition} becomes trivial. Our new result is concerned with the asymptotic analysis obtained under \eqref{main-condition} with $\alpha\neq\eta$. In particular, we have obtained in \cite{CapDovAsy} the following result in $\mathbb{R}^2$ for the elliptic problem \eqref{elliptic-asy} below.
\begin{theorem}
\label{main-thm}
Let $Y_t$ be a reflecting Brownian motion on $\overline{\Omega_\ell} \setminus \overline{\Omega_l}$ with boundary local time
\begin{align*}
\gamma^{\partial \Omega_\ell}_t (Y) := \int_0^t \mathbf{1}_{\partial \Omega_\ell} (Y_s)ds.
\end{align*}
Under \eqref{main-condition}, we have that
\begin{align}
\label{eq-main-thm}
\lim_{n \to \infty}\mathbb{E}_x \left[ \int_0^\infty f(B^{(\alpha, \eta)}_t) M^\epsilon_t dt \right] = \mathbb{E}_x\left[ \int_0^{\tau_l}  f(Y_t) \exp \left( - \left( \lim_{n \to \infty} \frac{\alpha}{(1-\alpha) \epsilon} \right) \gamma^{\partial \Omega_\ell}_t(Y) \right) dt \right]
\end{align}
where $M^\epsilon_t := \mathbf{1}_{(t < \tau_{\epsilon})}$ with $\tau_{\epsilon} := \inf \{s > 0\,:\, B^{(\alpha, \eta)}_s \notin \Omega_r \setminus \overline{\Omega_l}\}$ and $\tau_{l} = \inf \{s > 0\,:\, Y_s \in \partial \Omega_l\}$.
\end{theorem}
We are interested in the asymptotic analysis (as $n \to \infty$) of the solution 
\begin{align}
\label{elliptic-asy}
u \in D(A^n) \quad s.t.\ A^n u = -f
\end{align}
on the collapsing domain $\Omega_r$ under condition \eqref{main-condition}. We have the transmission condition
\begin{eqnarray}
u|_{\partial \Omega^{-}_\ell} &=& u|_{\partial \Omega^{+}_\ell} \qquad \textrm{continuity on the boundary}, \label{BC1}\\
(1-\alpha) \partial_{ \nu} u |_{\partial \Omega^{-}_\ell} &=& \alpha \partial_{\nu} u |_{\partial \Omega^{+}_\ell} \qquad \textrm{partial reflection}, \label{BC2}
\end{eqnarray}
where $\partial_{\nu}u$ is the normal derivative of $u$ and
\begin{align}
\label{dom-A-form}
D(A^n) =\{g, A^ng \in C_b\,:\, g|_{\partial \Omega_l} = g|_{\partial \Omega_r}=0,\, g\textrm{ satisfies } \eqref{BC1},\, \eqref{BC2}\}.
\end{align} 
The infinitesimal generator
\begin{equation}
\label{A-form}
A^nu= \left\lbrace
\begin{array}{ll}
\displaystyle \frac{1}{2}\Delta u & \textrm{on } \Omega_\ell \setminus \overline{\Omega_l}\\
\displaystyle \frac{\eta}{2}\Delta u & \textrm{on } \Omega_r \setminus \overline{\Omega_\ell}
\end{array}
\right.
\end{equation}
can be written as follows
\begin{equation}
\label{A-d-dim}
 -\frac{\sigma^2(x)}{2 a(x)}\nabla \left(a(x) \nabla u \right)=f \quad \textrm{in} \quad {\Omega_{\ell+\epsilon} \setminus \overline{\Omega_l}} 
\end{equation}
where
\begin{equation}
a(x) = (1-\alpha) \mathbf{1}_{\overline{\Omega_\ell} \setminus \overline{\Omega_l}}(x) + \alpha \mathbf{1}_{\Omega_{\ell+\epsilon} \setminus \overline{\Omega_\ell}}(x), \quad \alpha \in (0,1)
\end{equation}
and 
\begin{equation}
\rho(x)a(x)=\sigma^2(x) = \mathbf{1}_{\overline{\Omega_\ell} \setminus \overline{\Omega_l}}(x) + \eta \mathbf{1}_{\Omega_{\ell+\epsilon} \setminus \overline{\Omega_\ell}}(x), \quad \eta>0.
\end{equation}
We consider the measure $dm=1/\rho(x) dx$ (where we denoted by $dx$ the Lebesgue measure on $\mathbb{R}^d$) under the assumption that \eqref{main-condition} holds true. Let $F$ be an open regular domain  such that $F \supset \overline{\Omega_{r}}.$ We have studied in \cite{CapDovAsy} the Mosco convergence of the sequence of energy forms in $L^2(F)$
\begin{equation}
\mathcal{E}^n(u,u)=
\begin{cases}
\int_{\Omega_\ell \setminus \overline{\Omega_l}} (1-\alpha) |\nabla u|^2 dx+\alpha \int_{\Omega_{\ell+\epsilon} \setminus \overline{\Omega_\ell}} |\nabla u|^2 dx &\text{if} \, u|_{\Omega_r \setminus \overline\Omega_l}\in H^1_0(\Omega_r \setminus \overline{\Omega_l})\\
+ \infty &\text{otherwise in }  L^2(F)
\end{cases} \label{A(n)2}
\end{equation}
according with $\alpha/\epsilon \to c \in [0, \infty]$ as $n \to \infty$ (see \cite{CapDovAsy} for details, Theorem 6.1, Theorem 6.2, Theorem 6.3).

In the setting of the previous sections we have the sequence $X^n = B^{(\alpha, \eta)}$ on $E^n=\Omega_r \setminus \overline{\Omega_l}$ is a modified skew Brownian motion with the boundary conditions prescribed above. 

\begin{theorem}
Let the setting of Theorem \ref{main-thm} prevails. The solution to \eqref{time-frac-problem} with $(A^n, D(A^n))$ given in \eqref{A-d-dim} and \eqref{dom-A-form} is 
\begin{align*}
P^{\Phi, n}_t f(x) = \mathbf{E}_x[f(B^{(\alpha, \eta)}_{L_t}) M^\epsilon_{L_t}]
\end{align*}
and $P^{\Phi, n}_t \to P^\Phi_t$ strongly in $L^2(\Omega_\ell \setminus \overline{\Omega_l}, dx)$ as $n\to \infty$ where
\begin{align*}
P^\Phi_t f(x) = \mathbf{E}_x[f(X^\Phi_t)] = \mathbf{E}_x [f(Y^\Phi_t) \mathbf{1}_{(L_t < \zeta)} ]
\end{align*}
is the solution to \eqref{time-frac-problem} with $(A, D(A))$ where $A$ is the Dirichlet Laplacian if $\alpha/\epsilon \to \infty$, the Neumann Laplacian if $\alpha/\epsilon\to 0$, the Robin Laplacian if $\alpha/\epsilon \to c \in (0, \infty)$. Moreover, under \eqref{main-condition}, we have that
\begin{align}
\label{eq-main-thm-L}
\lim_{n \to \infty}\mathbb{E}_x \left[ \int_0^\infty f(B^{(\alpha, \eta)}_{L_t}) M^\epsilon_{L_t} dt \right] = \mathbb{E}_x\left[ \int_0^{\infty}  f(Y_t^\Phi) \exp \left( - \left( \lim_{n \to \infty} \frac{\alpha}{(1-\alpha) \epsilon} \right) \gamma^{\partial \Omega_\ell}_{L_t}(Y) \right) \mathbf{1}_{(L_t < \tau_l)} dt \right].
\end{align}
\end{theorem}

\begin{proof}
From Theorem \ref{time-frac-THM} we obtain the solution to \eqref{time-frac-problem} with $A^n$ and $A$. From Theorem \ref{main2} we obtain the asymptotic results. Formula \eqref{eq-main-thm-L} can be obtained by the convergence of the solutions.
\end{proof}
We have that (see also \eqref{semigPhi-L-continuity} and \eqref{semigPhi-1}),
\begin{align*}
\lim_{\lambda \downarrow 0} R^\Phi_\lambda \mathbf{1}(x) = \int_0^\infty \mathbf{P}_x(L_t < \zeta) dt = \int_0^\infty \mathbf{E}_x \left[ (\gamma^{\partial \Omega_\ell}_{L_t} < V), (L_t < \tau_l) \right] dt
\end{align*} 
where $V$ is an exponential random variable independent from $X^n$ and $L$ with $\mathbf{P}(V>v) = e^{-c v}$ and $c = \lim_{\epsilon \to 0} \alpha/\epsilon$. Then, we obtain
\begin{align*}
\displaystyle \lim_{\lambda \downarrow 0} R^\Phi_\lambda \mathbf{1}(x) = & \mathbf{E}_x \left[ \int_0^\infty e^{- \big(\lim_{\epsilon \to 0} \frac{\alpha}{\epsilon}\big)\gamma^{\partial \Omega_\ell}_{L_t}} \mathbf{1}_{(L_t < \tau_l)} dt \right]\\ 
\displaystyle = & \left( \lim_{\lambda \to 0} \frac{\Phi(\lambda)}{\lambda} \right) \mathbf{E}_x \left[ \int_0^{\infty} e^{- \big(\lim_{\epsilon \to 0} \frac{\alpha}{\epsilon}\big)\gamma^{\partial \Omega_\ell}_{L_t}} \mathbf{1}_{(L_t < \tau_l)} dt \right].
\end{align*}
We use the convention $\infty \cdot \gamma_t^{\partial \Omega_\ell} = 0$ for $t< \tau_\ell : = \inf\{s>0\,:\, Y_s \in \partial \Omega_\ell\}$. Then, in the case of Dirichlet condition ($\alpha/\epsilon \to \infty$)
\begin{align}
\mathbf{E}_x \left[ \int_0^{\infty} e^{- \big(\lim_{\epsilon \to 0} \frac{\alpha}{\epsilon}\big)\gamma^{\partial \Omega_\ell}_{L_t}} \mathbf{1}_{(L_t < \tau_l)} dt \right] = & \mathbf{E}_x\left[ \int_0^\infty \mathbf{1}_{(L_t < \tau_l \wedge \tau_\ell)} dt \right] \notag\\
= & \lim_{\lambda \downarrow 0} \frac{1}{\lambda} \mathbf{E}_x \left[ 1 - e^{-\Phi(\lambda) (\tau_l \wedge \tau_\ell)} \right] \label{Dir-1}\\
= & \Phi^\prime(0) \mathbf{E}_x [(\tau_l\wedge \tau_\ell)]. \notag
\end{align}
In the case of Neumann condition ($\alpha /\epsilon \to 0$)
\begin{align}
\mathbf{E}_x \left[ \int_0^{\infty} e^{- \big(\lim_{\epsilon \to 0} \frac{\alpha}{\epsilon}\big)\gamma^{\partial \Omega_\ell}_{L_t}} \mathbf{1}_{(L_t < \tau_l)} dt \right] = & \mathbf{E}_x\left[ \int_0^\infty \mathbf{1}_{(L_t < \tau_l)} dt \right] \notag \\
= & \lim_{\lambda \downarrow 0} \frac{1}{\lambda} \mathbf{E}_x \left[ 1 - e^{-\Phi(\lambda) \tau_l} \right] \label{Neu-1}\\
= & \Phi^\prime(0) \mathbf{E}_x [\tau_l]. \notag
\end{align}
Finally, Robin condition follows by considering $\alpha/\epsilon \to c \in (0,\infty)$: in particular, 
\begin{align}
\mathbf{E}_x \left[ \int_0^{\infty} e^{- c\gamma^{\partial \Omega_\ell}_{L_t}} \mathbf{1}_{(L_t < \tau_l)} dt \right] = & \mathbf{E}_x\left[ \int_0^\infty \mathbf{1}_{(L_t < \tau_l)} dt \right] \notag \\
= & \lim_{\lambda \downarrow 0} \frac{\Phi(\lambda)}{\lambda} \mathbf{E}_x \left[ \int_0^{\tau_l} e^{- s \Phi(\lambda)- c \gamma_s^{\partial \Omega_\ell}} ds \right] \label{Rob-1}\\
= & \Phi^\prime(0) \mathbf{E}_x \left[ \int_0^{\tau_l} e^{- c \gamma_s^{\partial \Omega_\ell}} ds \right]. \notag
\end{align}
Formulas \eqref{Dir-1}, \eqref{Neu-1} and \eqref{Rob-1} can be obtained from Theorem \ref{thm:lpot-L}.

\subsection{Further examples}
The previous example is given for a regular domain. A similar example can be given  for domains with fractal boundaries as  in \cite{CapDovPota} where   the authors have obtained asymptotic results for skew Brownian diffusions across Koch interfaces by using $M-$convergence results proved in \cite{C-JMAA}, \cite{CV-JDE}, \cite{CV-asy}.

We recall that  $M-$convergence results have been obtained on fractal structures in order to study several  boundary value problems
 (\cite{AD}, \cite{CLV-DIE}, \cite{Lancia-DCDSS}, \cite{LVER}, \cite{LV}),
reinforcement problems for variational inequalities (\cite{CV-CVPDE}),
dynamical quasi-filling fractal layers, layered fractal fibers and potentials (\cite{CV-SIAM}, \cite{MV1}, \cite{MV2}).

Moreover, we point out that $M-$convergence results have been obtained also for non-local Dirichlet forms (see, for example, \cite{BBCK} and \cite{CKK}) and then we can apply the theory developed in the present paper also in this framework (for example in the asymptotic study of jump-processes).

Finally, we remark that  Theorems \ref{main1} and  \ref{main2}   allows us to obtain  asymptotic  results for fractional equations  also via $\Gamma$-convergence if the corresponding  forms are asymptotically compact (see, for some examples and applications,  \cite{BRA} and \cite{DALMASO}).

\end{document}